\documentclass[12pt]{amsart}
\usepackage{amsmath,mathtools,amsthm,amsfonts,bigints,amssymb, mathrsfs}
\usepackage{xcolor}
\RequirePackage[bookmarks, bookmarksopen=true, plainpages=false, pdfpagelabels, pdfpagelayout=SinglePage, breaklinks = true]{hyperref}
\usepackage[noabbrev,capitalize,nameinlink]{cleveref}
\hypersetup{
	colorlinks,
	linkcolor={blue!50!black},
	citecolor={green!50!black},
	urlcolor={blue!50!black}
}

\makeatletter
\def\author@andify{%
	\nxandlist {\unskip ,\penalty-1 \space\ignorespaces}%
	{\unskip {} \@@and~}%
	{\unskip \penalty-2 \space \@@and~}%
}

\makeatletter
\@namedef{subjclassname@2020}{\textup{2020} Mathematics Subject Classification}
\makeatother

\newtheorem{theorem}{Theorem}[section]
\newtheorem{lemma}[theorem]{Lemma}
\newtheorem{proposition}[theorem]{Proposition}

\newtheorem{remark}[theorem]{Remark}

\newcommand{\R}{\mathbb R}%
\numberwithin{equation}{section}

\begin{document}
	\title[Poisson kernel for Hadamard manifolds]{Estimates of the Poisson kernel on negatively curved Hadamard manifolds }
	\author[Biswas]{Kingshook Biswas}
	\address{Stat-Math Unit, Indian Statistical Institute, 203 B. T. Rd., Kolkata 700108, India}
	\email{kingshook@isical.ac.in}
	\author[Dewan]{Utsav Dewan}
	\address{Stat-Math Unit, Indian Statistical Institute, 203 B. T. Rd., Kolkata 700108, India}
	\email{utsav\_r@isical.ac.in}
	\author[Pal Choudhury]{Arkajit Pal Choudhury}
	\address{Stat-Math Unit, Indian Statistical Institute, 203 B. T. Rd., Kolkata 700108, India}
	\email{rkjtpalchoudhury\_r@isical.ac.in}
	\subjclass[2020]{Primary 53C20; Secondary 31C05.} 
	\keywords{Hadamard manifold, Poisson kernel, Harmonic measure, Harmonic functions}
	
	\begin{abstract} 
	Let $M$ be an $n$-dimensional Hadamard manifold of pinched negative curvature $-b^2 \leq K_M \leq -a^2$. The solution of the Dirichlet problem at infinity for $M$ leads 
     to the construction of a family of mutually absolutely continuous probability measures $\{\mu_x\}_{x \in M}$ called the harmonic measures. Fixing a basepoint 
$o \in M$, the Poisson kernel of $M$ is the function $P : M \times \partial M \to (0, \infty)$ defined by
$$
P(x, \xi) = \frac{d\mu_x}{d\mu_o}(\xi) \ , \ x \in M, \xi \in \partial M.
$$
 We prove the following global upper and lower bounds for the Poisson kernel:
$$
\frac{1}{C}\: e^{-2K{(o|\xi)}_x}\: e^{a d(x, o)} \le P(x,\xi) \le C\: e^{2K{(x|\xi)}_o}\: e^{-a d(x,o)} \:		
$$
for some positive constants $C \geq 1, K > 0$ depending solely on $a, b$ and $n$. 

\medskip

The above estimates may be viewed as a generalization of the well-known formula for the Poisson kernel in terms of Busemann functions for the special case 
of Gromov hyperbolic harmonic manifolds. These estimates do not follow directly from known estimates on Green's functions or harmonic measures. Instead we use techniques due to Anderson-Schoen 
for estimating positive harmonic functions in cones. As applications, we obtain quantitative estimates for the convergence $\mu_x \to \delta_{\xi}$ as $x \in M \to \xi \in \partial M$, 
and for the convergence of harmonic measures on finite spheres to the harmonic measures on the boundary at infinity as the radius of the spheres tends to infinity.
	\end{abstract}
	\maketitle
	\tableofcontents
	\section{Introduction}
	Let $M$ be a Hadamard manifold, i.e. a complete, simply connected Riemannian manifold of nonpositive sectional curvature, and let $\partial M$ denote its visual boundary. 
Assume furthermore that $M$ has pinched negative curvature $-b^2 \leq K_M \leq -a^2$ for some constants $b \geq a > 0$. 
There is a natural class of probability measures on the boundary $\partial M$ called the {\it harmonic measures}, which arise from the solution of the {\it Dirichlet problem at infinity}. 
The Dirichlet problem at infinity consists in finding, for any continuous function $f$ on $\partial M$, a harmonic function $u$ on $M$ which is continuous on 
$\overline{M} = M \cup \partial M$ and with boundary value equal to $f$, $u_{|\partial M} = f$. This problem was solved independently by Sullivan \cite{Sullivan}, using 
probabilistic techniques, and by Anderson \cite{Anderson} using subharmonic barriers and the Perron method. Denoting by $u_f$ the unique solution to the Dirichlet problem with 
boundary value $f$, the harmonic measures are a family of probability measures $\{ \mu_x \}_{x \in M}$ on $\partial M$ characterized by the equality

\begin{equation}\label{harmonicmeasures}
u_f(x) = \int_{\partial M} f(\xi) \ d\mu_x(\xi) \ \ \ \hbox{for all} \ \ x \in M, f \in C(\partial M).
\end{equation}

One can show that the harmonic measures $\mu_x, x \in M$ are mutually absolutely continuous. Fixing a basepoint $o \in M$, the {\it Poisson kernel} of $M$ is the function 
$P : M \times \partial M \to (0, \infty)$ defined by
$$
P(x, \xi) = \frac{d\mu_x}{d\mu_o}(\xi) \ , \ x \in M, \xi \in \partial M.
$$

For noncompact rank one symmetric spaces, there is a well-known formula for the Poisson kernel in terms of the {\it Busemann cocycle} of the space, which shows 
in particular that the Poisson kernel can be expressed entirely in terms of the distance function in this case. Here the Busemann cocycle of a negatively curved 
Hadamard manifold $M$ is the function $B : M \times M \times \partial M \to \mathbb{R}$ defined by 
$B(x,y,\xi) = \lim_{z \to \xi} (d(x,z) - d(y,z))$. Then the Poisson kernel for rank one symmetric spaces is given by 
\begin{equation} \label{poissonbusemann}
P(x, \xi) = e^{hB(o,x,\xi)}
\end{equation}
where $h > 0$ is the logarithmic volume growth of the manifold, given by 
$$
h = \lim_{R \to \infty} \frac{\log(\hbox{vol}(B(o,R)))}{\log R}.
$$ 
More generally, Knieper and Peyerimhoff showed that formula (\ref{poissonbusemann}) for the Poisson kernel holds for all rank one 
asymptotically harmonic manifolds (\cite{KnPe}). This class of manifolds includes in particular all the known examples of nonflat noncompact harmonic manifolds, 
namely the noncompact rank one symmetric spaces and the Damek-Ricci spaces.

\medskip

For general negatively curved Hadamard manifolds it is too much to expect an exact formula like (\ref{poissonbusemann}) in terms of only the distance function to hold. 
Our aim in this article is to instead prove in the general case global upper and lower bounds for the Poisson kernel in terms of the distance function analogous to 
(\ref{poissonbusemann}).  We first note that the Busemann cocycle can be expressed in terms of the Gromov inner product as
$$
B(o, x, \xi) = d(o,x) - 2(o|\xi)_x = -d(o,x) + 2(x|\xi)_o
$$
and hence we have the following expressions for the Poisson kernel in the case of rank one asymptotically harmonic manifolds:
\begin{equation} \label{poissonharmonic}
e^{-2h(o|\xi)_x} e^{hd(o,x)} = P(x, \xi) = e^{2h(x|\xi)_o} e^{-hd(o,x)}.
\end{equation}

We prove the following upper and lower bounds for the Poisson kernel, which may be viewed as generalizations of the two expressions above:


\medskip

	\begin{theorem} \label{Poisson_estimate_thm}
Let $0 < a \le b$ and $n \ge 2$. There exist positive constants $C \geq 1, K > 0$ depending solely on $a, b$ and $n$ such that for any $n$-dimensional 	Hadamard manifold $M$ with pinched curvature $-b^2 \le K_M \le -a^2$, with origin $o$, the Poisson kernel of $M$ satisfies
		\begin{equation} \label{Poisson_estimate}
		\frac{1}{C}\: e^{-2K{(o|\xi)}_x}\: e^{a d(o, x)} \le P(x,\xi) \le C\: e^{2K{(x|\xi)}_o}\: e^{-a d(o, x)} \:,
		\end{equation} 
		for all $x \in M\:,\:\xi \in \partial M$.
	\end{theorem}

\medskip

We remark that our estimates in (\ref{Poisson_estimate}) lead in the general case to the same two well-known conclusions that one obtains from the expressions in 
(\ref{poissonharmonic}), namely exponential decay of $P(x,\xi)$ as $x \to \infty$ outside any cone neighbourhood of $\xi$, and exponential growth of $P(x,\xi)$ 
as $x \to \xi$ inside any $\epsilon$-neighbourhood of any geodesic ray with endpoint $\xi$.

\medskip
   
Since the Poisson kernel is defined via Radon-Nikodym derivatives of harmonic measures, one may hope that estimates on harmonic measures would lead to 
estimates on the Poisson kernel. Indeed, it follows from Anderson-Schoen \cite{AS}, in particular section 4 of \cite{AS}, that for any $\xi \in \partial M$ and any 
descending sequence of angles $\theta_i$ converging to $0$, if $C_i = C(o, \xi, \theta_i) \subset \overline{M}$ denotes the solid cone with vertex at $o \in M$, centered around 
$\xi \in \partial M$ and of angle $\theta_i$, then the Poisson kernel can be written as a limit of ratios of harmonic measures:

\begin{equation} \label{ratiolimit}
P(x, \xi) = \lim_{i \to \infty} \frac{\mu_x(C_i)}{\mu_o(C_i)}.
\end{equation}

While there has been previous work on estimates for harmonic measures on negatively curved manifolds, including the work of Kifer-Ledrappier \cite{Kifer} and 
Benoist-Hulin \cite{BeHu}, these estimates on harmonic measures are too weak to prove directly the above estimates on the Poisson kernel. Benoist-Hulin prove the following 
estimate on harmonic measures of cones, 
$$
\frac{1}{C} \theta^K \leq \mu_x(C(x, \xi, \theta)) \leq C \theta^{1/K}
$$
where $C, K > 0$ are constants depending only on $a, b$ and $n$. However the difference between the exponents of $\theta$ in the upper and lower bounds in this estimate 
means that it gives no information on the Poisson kernel when trying to compute it as in (\ref{ratiolimit}) above as a limit of ratios of harmonic measures.  

\medskip

As shown by Anderson-Schoen \cite{AS}, the Poisson kernel may also be computed as a limit of ratios of Green's functions:

\begin{equation} \label{ratiogreens}
P(x, \xi) = \lim_{y \to \xi} \frac{G(y, x)}{G(y, o)}
\end{equation}

where for $y \in M$, $G(y, .)$ denotes the Green's function for $M$ with pole at $y$. One has the following estimates due to Ancona \cite{Ancona} for the Green's 
functions,
$$
\frac{1}{C} e^{- \lambda_1 d(y,x)} \leq G(y,x) \leq C e^{-\lambda_2 d(y,x)} \ \ \ \hbox{for all} \ \ x \in M - B(y, 1/2)
$$
for some constants $C > 0$ and $\lambda_1 = (n-1)\cdot b\ \hbox{coth}(b/4) > \lambda_2 = (n-1)\cdot a$. Again, the fact that in the above upper and lower bounds the exponents 
$\lambda_1, \lambda_2$ differ, means that one cannot use the above estimates to estimate the Poisson kernel via the formula (\ref{ratiogreens}) above. 

\medskip

\medskip

The main tool to prove our estimates is a generalization of a result due to Anderson-Schoen \cite{AS} on the exponential decay near infinity of positive 
harmonic functions $h$ defined in a cone $C$ which vanish at infinity in the cone, i.e. $h(x) \to 0$ as $x \in C$ tends to infinity. Applying this result to the 
functions $h(x) = P(x, \xi)$ for appropriately chosen cones $C$ depending on $\xi$ allows us to prove the estimates of Theorem \ref{Poisson_estimate_thm}. 
   
\medskip

As pointed out by Krantz in \cite{Krantz}, the size estimates of the Poisson kernel play a fundamental role in studying the boundary behavior of harmonic functions and many other aspects of potential theory. In this article however, we present two applications of our estimates of a slightly different flavour. 

\medskip

The first one is regarding the 
concentration of the harmonic measures $\mu_x$ as $x \in M$ converges to a boundary point $\xi \in \partial M$. It is easy to see from the characterization 
(\ref{harmonicmeasures}) of harmonic measures that in this case the measures $\mu_x$ converge weakly to the Dirac mass $\delta_{\xi}$ based at $\xi$. We give a 
quantitative version of this convergence using the estimates from Theorem \ref{Poisson_estimate_thm}, in terms of the rate of decay of harmonic measures 
of complements of cone neighbourhoods of $\xi$. 

\medskip

The second application concerns the convergence 
of harmonic measures $\mu_{o,R}$ on finite spheres $S(o,R)$ to the harmonic measure $\mu_o$ on $\partial M$. While it is possible to show by relatively soft 
arguments the weak convergence 
of the measures $\mu_{o,R}$ to the measure $\mu_o$ as $R \to \infty$ (after identifying the measures $\mu_{o,R}$ with measures on $\partial M$ via radial projection from 
$S(o,R)$ to $\partial M$), we use the estimates of Theorem \ref{Poisson_estimate_thm} to obtain again a quantitative version of this convergence in terms 
of the rate of convergence of integrals of Holder functions on $\partial M$ against these measures. 

\medskip

The paper is organized as follows. In section $2$, we discuss the relevant preliminaries and fix our notation. In section $3$, we obtain an auxiliary estimate for positive harmonic functions. In section $4$, we present some important inequalities involving Riemannian angles and Gromov products. The lower bound and the upper bound of the Poisson kernel are proved in sections $5$ and $6$ respectively. Finally, we conclude by presenting some applications of the pointwise estimates of the Poisson kernel in section $7$. 

\section{Preliminaries}
We will use $c_1, c_2, c_3, \dots$ to denote positive constants depending only on $n,a,b$. 

\medskip

We first recall briefly some basic properties of Gromov hyperbolic spaces, for more details we refer to \cite{Bridson}. A {\it geodesic} in a metric space $M$ is an isometric embedding $\gamma : I \subset \mathbb{R} \to M$ of an interval into $M$. The metric space $M$ is said to be {\it geodesic} if any two points in $M$ can be joined by a geodesic. A geodesic metric space $M$ is said to be {\it Gromov hyperbolic} if there is a $\delta \geq 0$, such that every geodesic triangle in $M$ is $\delta$-thin, i.e. each side is contained in the $\delta$-neighbourhood of the union of the other two sides.

\medskip

The {\it Gromov boundary} of a Gromov hyperbolic space $M$ is defined to be the set $\partial M$ of equivalence classes of geodesic rays in $M$. Here a geodesic ray is an isometric embedding $\gamma : [0,\infty) \to M$ of a closed half-line into $M$, and two geodesic rays $\gamma_1, \gamma_2$ are said to be equivalent if the set $\{ d(\gamma_1(t), \gamma_2(t)) \ | \ t \geq 0 \}$ is bounded. The 
equivalence class of a geodesic ray $\gamma$ is denoted by $\gamma(\infty) \in \partial M$. 

\medskip

We now assume $M$ is a Hadamard manifold satisfying the hypothesis of Theorem \ref{Poisson_estimate_thm}. Then $M$ is, in particular, a Gromov hyperbolic space and hence is equipped with a Gromov boundary $\partial M$. We fix an origin $o$ and for $x \in M$, let $\rho(x)$ denote the distance from $x$ to $o$. Let $B(x,R)$ and $S(x,R)$ denote respectively the geodesic ball and the geodesic sphere with center $x \in M$ and radius $R>0$, respectively. For every geodesic ray $\gamma$,  $\gamma(t) \to \gamma(\infty) \in \partial M$ as $t \to \infty$, and for any $x \in M, \xi \in \partial M$, there exists a geodesic ray  
$\gamma$ such that $\gamma(0) = x, \gamma(\infty) = \xi$. 

\medskip

There is a natural topology on $\overline{M} := M \cup \partial M$, called the {\it cone topology} such that $\overline{M}$ is a compact metrizable space which is a compactification of $M$. This topology is defined as follows: for $v \in T_{x}M$, let $C(x,\alpha)$ be the cone with vertex $x$ and aperture $\alpha$, that is,
$$
C(x,\alpha)=\{y \in M : \angle_{x}\left(v,T_{xy}\right)< \alpha\}\:,
$$ 
where $T_{xy}$ is the tangent vector to the geodesic ray through $x$ and $y$ and $\angle_x$ denotes the angle in $T_xM$. Let 
$$T(x,\alpha,R)=C(x,\alpha)\setminus B(x,R)\:,$$
denote a truncated cone. Then for all such $v$, the domains $T(o,\alpha,R)$ together with the geodesic balls $B(q,r)$, $q \in M$, form a local basis for the cone topology. We will refer to neighbourhoods of $\xi \in \partial M$ with respect to the cone topology as {\it cone neighbourhoods} and denote them by $\mathscr{C}(\xi,\cdot)$. Also when, the vertex of a cone or a truncated cone is $o$, for notation convenience, we will simply omit the vertex in the notation. $\partial C(x,\alpha)$ will denote the boundary of the cone $C(x,\alpha)$ and $\partial T(x,\alpha,R):=\partial C(x,\alpha) \setminus B(x,R)$.

\medskip

For three points $x,y,z \in M$, the {\it Gromov inner product} of $y,z$ 
with respect to $x$ is defined to be   
$$
(y|z)_x := \frac{1}{2}(d(x,y)+d(x,z) - d(y,z)).
$$
Moreover, in our case for $x \in M$, the Gromov inner product $(.|.)_x : M \times M \to [0,+\infty)$ extends to a continuous function $(.|.)_x : \overline{M} \times \overline{M} \to [0,+\infty]$, such that $(\xi|\eta)_x = +\infty$ if and only if $\xi = \eta \in \partial M$. In fact the compactification $\overline{M}$ is homeomorphic to the closed unit ball $\overline{\mathbb{B}} \subset \mathbb{R}^n$, and there is a homeomorphism 
$\phi : \overline{\mathbb{B}} \to \overline{M}$ such that the restriction to the open unit ball $\phi : \mathbb{B} \to M$ is a diffeomorphism.

\medskip

The {\it Busemann cocycle} of $M$ is the function $B : M \times M \times \partial M \to \mathbb{R}$ defined by
$$
B(x, y, \xi) := \lim_{z \to \xi} (d(x,z) - d(y,z)).
$$
The limit above exists, and the Busemann cocycle is a continuous function on $M \times M \times \partial M$. 

\medskip

Let $\Delta$ be the Laplace-Beltrami operator on $M$, corresponding to the Riemannian structure on $M$. We next discuss the important notion of harmonic measures. These arise from the solution of the Dirichlet problem. Since $M$ is negatively curved, any geodesic sphere in $M$ is a $C^\infty$-submanifold of $M$ and hence the Dirichlet problem is solvable on geodesic balls. Then for any $y \in B(x,R)$, the {\it harmonic measure on $S(x,R)$ with respect to $y$} is the probability measure $\mu_{y, B(x,R)}$ on $S(x,R)$ defined by
$$
\int_{S(x,R)} f \ d\mu_{y, B(x,R)} = u_f(y)
$$
for all continuous functions $f$ on $S(x,R)$, where $u_f$ is the solution of the Dirichlet problem in $B(x,R)$ with boundary value $f$.  The harmonic measures $\mu_{y, B(x,R)}$ are mutually absolutely continuous, in fact they are absolutely continuous with respect to the Lebesgue measure of $S(x,R)$. The harmonic measure $\mu_{y, B(x,R)}$ can also be described in terms of Brownian motion started at $y$. If $(B_t)_{t \geq 0}$ is a Brownian motion started at $y$, and $\tau$ is defined to be the first exit time from $B(x,R)$, i.e. 
$$
\tau = \inf\{ \ t > 0 \ | \ B(t) \notin B(x,R) \ \},
$$
then we have
$$
\int_{S(x,R)} f \ d\mu_{y, B(x,R)} = \mathbb{E}(f(B_{\tau}))
$$
for all continuous functions $f$ on $S(x,R)$. In this article, we will only be interested in the harmonic measure with respect to the center $x$ of $B(x,R)$, and thus we simply write $\mu_{x,R}$ in stead of $\mu_{x,B(x,R)}$.
 
\medskip

For a manifold $M$ as in our consideration, it is well-known that the Dirichlet problem at infinity is solvable (\cite{AS}, \cite{Sullivan}), and so in this case one can also define a family of harmonic measures $\{ \mu_x \}_{x \in M}$, which are probability measures on $\partial M$ defined by
$$
\int_{\partial M} f \ d\mu_x = u_f(x)
$$
for any continuous function $f$ on $\partial M$, where $u_f$ is the solution of the Dirichlet problem at infinity with boundary value $f$, i.e. $u_f$ is harmonic on $M$ and $u_f(x) \to f(\xi)$ as $x \to \xi \in \partial M$, for any $\xi \in \partial M$. As in the case of a bounded domain, the harmonic measures $\mu_x$ are mutually absolutely continuous. 

\medskip
 
But due to variable curvature, the harmonic measures at infinity need not belong to the same measure class as the geodesic measure class on $\partial M$. As before, in this case also the harmonic measure $\mu_x$ can be described in terms of Brownian motion started at $x$. If $(B(t))_{t \geq 0}$ is a Brownian motion started at $x$, then it is known (\cite{Sullivan}) that almost every sample path of the Brownian motion converges to a (random) point $B_{\infty}$ in $\partial M$. This limiting point $B_{\infty}$ is a random variable taking values in $\partial M$, whose distribution is precisely the harmonic measure $\mu_x$; for any continuous function $f$ on $\partial M$, we have
$$
\int_{\partial M} f \ d\mu_x = \mathbb{E}(f(B_{\infty})).
$$

\medskip

The Poisson kernel $P(\cdot,\cdot)$ is jointly continuous on $M \times \partial M$ and in addition, for any fixed $\xi \in \partial M$, $P(\cdot,\xi)$ is a positive harmonic function in $M$ that vanishes at all $\eta \ne \xi \in \partial M$ and has a singularity precisely at $\xi$.


\medskip
	
	An important tool in estimating positive harmonic functions is the Harnack-Yau inequality. We state without proof a version of it below.
	\begin{lemma} \label{Har-Yau_lemma}
		Let $M$ be a Hadamard manifold with $-b^2 \le K_M \le 0$ with dimension $\ge 2$. Then, there exists a constant $c_1 = c_1(n, b)$, such that for every open set $\Omega \subset M$ and every positive harmonic function $u$ on $\Omega$, one has
		\begin{equation*} 
		\|\nabla \log u(x)\| \le c_1 \:, \text{ for all } x \in M \text { with } d(x,\partial \Omega) \ge 1 \:.
		\end{equation*}
	\end{lemma}
	This lemma is in fact true for any complete Riemannian manifold whose Ricci curvature is bounded below. A short proof can be found in \cite[Lemma 2.1]{LW}.

\medskip

The pinching condition on the sectional curvatures is instrumental in estimating  angles via useful comparison results. If three points in $M$ lie on the same geodesic, then they are called {\it collinear}. For three points $x,y,z$ which are not collinear, we form the geodesic triangle $\triangle$ in $M$ by the geodesic segments $[x, y],\: [y, z],\: [z, x]$. A comparison triangle is a geodesic triangle $\overline{\triangle}$ in $\mathbb{H}^2(-b^2)$ formed by geodesic segments $[\overline{x}, \overline{y}],\: [\overline{y}, \overline{z}],\: [\overline{z}, \overline{x}]$ of the same lengths as those of $\triangle$ (such a triangle exists and is unique up to isometry). Let $\theta(y,z)$ denote the Riemannian angle between the points $y$ and $z$, subtended at $x$. The corresponding angle between $\overline{y}$ and $\overline{z}$ subtended at $\overline{x}$ is called the {\it comparison angle} of $\theta(y,z)$ in $\mathbb{H}^2(-b^2)$ and is denoted by $\theta_b(y,z)$. Then by Alexandrov's angle comparison theorem,
\begin{equation*} 
\theta_b(y,z) \le \theta(y,z) \:.
\end{equation*}
Now using the upper bound on the sectional curvature ($K_X \le -a^2$), one can similarly consider the comparison angles in $\mathbb{H}^2(-a^2)$, with the above inequality reversed.

\medskip

Our next Lemma asserts the exponential decay of the angle subtended by balls of constant radius, which are away from the origin. This is a standard result in geometry and can be found in \cite[p. 437]{AS}.
\begin{lemma}\label{ball_angle_lemma}
There exists a positive constant $c_2$ such that for $x \in M$, the angle subtended by the ball $B(x,\sqrt{2})$ at $o$, say $\omega(\rho(x))$, satisfies
		\begin{equation*}
		\omega(\rho(x)) \le c_2 e^{-a \rho(x)} \:.
		\end{equation*}
	\end{lemma}
	
\medskip

For $x,y \in \overline{M} \setminus \{o\}$, let $\angle_o(x,y)$ denote the Riemannian angle between $x$ and $y$, subtended at $o$. Then $\left(S(o,1),\angle_o\right)$ is a metric space of diameter $\pi$.	Let $\varphi$ be a Lipschitz function of the Riemannian angle on $S(o,1)$. Then one has the  Lipschitz semi-norm of $\varphi$:
\begin{equation*}
{|\varphi|}_{Lip}:= \displaystyle\sup_{\substack{x \ne y\\ x,y \in S(o,1)}} \frac{|\varphi(x)-\varphi(y)|}{\angle_o(x,y)}\:. 
\end{equation*}
Let $\overline{\varphi}$ be the extension of $\varphi$ on $\overline{M} \setminus \{o\}$ along radial geodesic rays emanating from $o$, with boundary values $\varphi$ on $\partial M$, as $\partial M$ can be identified with $S(o,1)$ under the natural radial projection. For such a function, one has the important notion of a convolution, introduced by Anderson and Schoen  \cite[p. 436]{AS}: let $\chi: \R \to (0,\infty)$ be a fixed $C^2$ approximation to the characteristic function of $[0,1]$ with $supp \:\chi \subset [-2,2]$. Then define an average of $\varphi$ with respect to $\chi$ by,
$$\mathscr{S}(\varphi)(x)= \frac{\int_M \chi\left({d(x,y)}^2\right) \overline{\varphi}(y)\: dy}{\int_M \chi\left({d(x,y)}^2\right) dy}\:.$$
It satisfies the following estimates:

\begin{lemma}\label{AS_convolution_lemma}
For $x \in M$, let $\eta_x$ denote the unique point on $\partial M$ obtained by extending the geodesic segment joining $o$ to $x$. There exists a positive constant $c_3$ such that for all $x \in M$, one has
\begin{itemize}
\item[(i)] $\left|\mathscr{S}(\varphi)(x) - \varphi(\eta_x)\right| \le c_3 \: {|\varphi|}_{Lip}\: e^{-a\rho(x)}$ \:,
\item[(ii)] $\|\nabla \mathscr{S}(\varphi)(x)\| \le c_3 \: {|\varphi|}_{Lip}\: e^{-a\rho(x)}$ \:,
\item[(iii)] $\|\nabla^2 \mathscr{S}(\varphi)(x)\| \le c_3 \:{|\varphi|}_{Lip}\: e^{-a\rho(x)}$ \:.
\end{itemize}
\end{lemma} 
\begin{proof}
The proof of $(i)$ is a simple consequence of Lemma \ref{ball_angle_lemma}. Indeed,
\begin{eqnarray*}
\left|\mathscr{S}(\varphi)(x) - \varphi(\eta_x)\right| & \le & \frac{\int_M \chi\left({d(x,y)}^2\right) |\overline{\varphi}(y)-\overline{\varphi}(x)|\: dy}{\int_M \chi\left({d(x,y)}^2\right) dy} \\
& \le & {|\varphi|}_{Lip} \: \omega(\rho(x)) \\
& \le & c_4 \: {|\varphi|}_{Lip}\: e^{-a\rho(x)} \:.
\end{eqnarray*}

The estimates on the covariant derivatives of $\mathscr{S}(\varphi)$ ((ii) and (iii)) can be proved similarly. We refer to \cite[pp. 436-437]{AS} for the computation.
\end{proof}

One can also consider functions $\varphi$ on $\partial M$. Then $\left(\partial M,\angle_o\right)$ is again a metric space of diameter $\pi$. However $\partial M$ may only have a $C^\beta$-H\"older structure for $\beta=a/b \in (0,1]$. Then one has the following $C^\beta$-H\"older semi-norms and norms of $\varphi$ respectively:
\begin{eqnarray*}
{|\varphi|}_{C^\beta}&:=& \displaystyle\sup_{\substack{\xi \ne \eta\\ \xi, \eta \in \partial M}} \frac{|\varphi(\xi)-\varphi(\eta)|}{\left(\angle_o(\xi,\eta)\right)^{\beta}}\:, \\
{\|\varphi\|}_{C^\beta}&:=& {|\varphi|}_{C^\beta} + {\|\varphi\|}_{\infty}\:.
\end{eqnarray*}

\section{An auxiliary estimate for positive harmonic functions}
In this section, we see some estimates on the behaviour of positive harmonic functions in a cone that vanish continuously at infinity. We first cite a result by Anderson and Schoen, in this direction.
	\begin{lemma} \cite[Theorem 4.1]{AS} \label{ASthm4.1}
		Let $\theta_0 \in (0,\pi)$ and let $h$ be a positive harmonic function in the cone $C(\theta_0)$, which is continuous in the closure of $C(\theta_0)$ and which vanishes on $\overline{C(\theta_0)} \cap \partial M$. Then there exists a positive constant $c_1$ such that
		\begin{equation} \label{ASthm4.1estimate}
		\displaystyle\sup_{T\left(\theta_0/2,R_0\right)} h \le \:\:c_1 \displaystyle\sup_{\partial B_0(R_0) \cap C\left(7\theta_0/8\right)} h \:, 
		\end{equation}
		where 
		\begin{equation} \label{R0in4.1}
		R_0= c_2 \log (1/\theta_0) + c_3 \:,
		\end{equation}
		for some positive constants $c_2,c_3$\:.
	\end{lemma}
	Applying Lemma \ref{ASthm4.1}, we get a pointwise estimate for positive harmonic functions (vanishing at infinity) in a suitable cone, in terms of the distance and the aperture of the cone. This is essentially a generalization of  \cite[Corollary 4.2]{AS}, as we deal with general apertures.  
\begin{lemma} \label{estimate_lemma}
First choose and fix $x_0 \in M$. Let $\theta_0 \in (0,\pi)$ and let $h$ be a positive harmonic function in $C(x_0, \theta_0)$ which is continuous in the closure of $C(x_0, \theta_0)$ and vanishes on $\overline{C(x_0, \theta_0)} \cap \partial M$. Then there exist positive constants $c_4$ and  $c_5$ such that
\begin{equation}
h(x) \le  c_4 {\left(\frac{1}{\theta_0}\right)}^{c_5} e^{-a d(x,x'_0)} h(x'_0) \:,
\end{equation}
for all $ x \in T(x_0,\theta_0/8, 1)$, where $x'_0=\exp_{x_0}(v_0)$ and $v_0$ is the axis vector of the cone, lying on $S(x_0,1)$.
\end{lemma}

\begin{proof}

Without loss of generality, we will prove Lemma \ref{estimate_lemma} for $x_0=o$. By Lemma \ref{ASthm4.1}, we have
\begin{equation}\label{Lemma_pf_eq1}
\displaystyle\sup_{T\left(\theta_0/4,R_0\right)} h \le \:\:c_6 \displaystyle\sup_{\partial B_0(R_0) \cap C\left(7\theta_0/16\right)} h \:, 
\end{equation}
where 
\begin{equation} \label{Lemma_pf_eq2}
R_0=c_7\log(2/\theta_0)+c_8\:.
\end{equation}
Now by Harnack-Yau (Lemma \ref{Har-Yau_lemma}), 
\begin{equation}\label{Lemma_pf_eq3}
\displaystyle\sup_{T(\theta_0/4,1) \cap \overline{B(o,R_0)}} h \le \:\: e^{c_9(R_0+1)}  h(x'_0)\:.
\end{equation}
Hence combining (\ref{Lemma_pf_eq1}) and (\ref{Lemma_pf_eq3}) we have
\begin{equation}\label{Lemma_pf_eq4}
\sup_{T(\theta_0/4,1)} h \le c_{10}\:e^{c_9R_0} h(x'_0)\:.
\end{equation}

\medskip

Next let $\theta$ denote the Riemannian angle in the unit sphere $S(o,1)$ measured from $v_0$ (the axis vector of the cones) and let $\varphi$ be a Lipschitz function of $\theta$ such that 
\begin{equation*}
\varphi(\theta)= \begin{cases}
0  &\text{if } \theta\le \frac{\theta_0}{8},\\
	\frac{\theta}{\theta_0/8}-1  &\text{if } \frac{\theta_0}{8}\le\theta\le\frac{\theta_0}{4},\\
1  &\text{if } \theta\ge \frac{\theta_0}{4}.
\end{cases}
\end{equation*}
We note that 
\begin{equation*}
{|\varphi|}_{Lip} \le c_{11}\theta^{-1}_0\:.
\end{equation*}

\begin{figure}
	\centering
	\includegraphics[width=.7\linewidth]{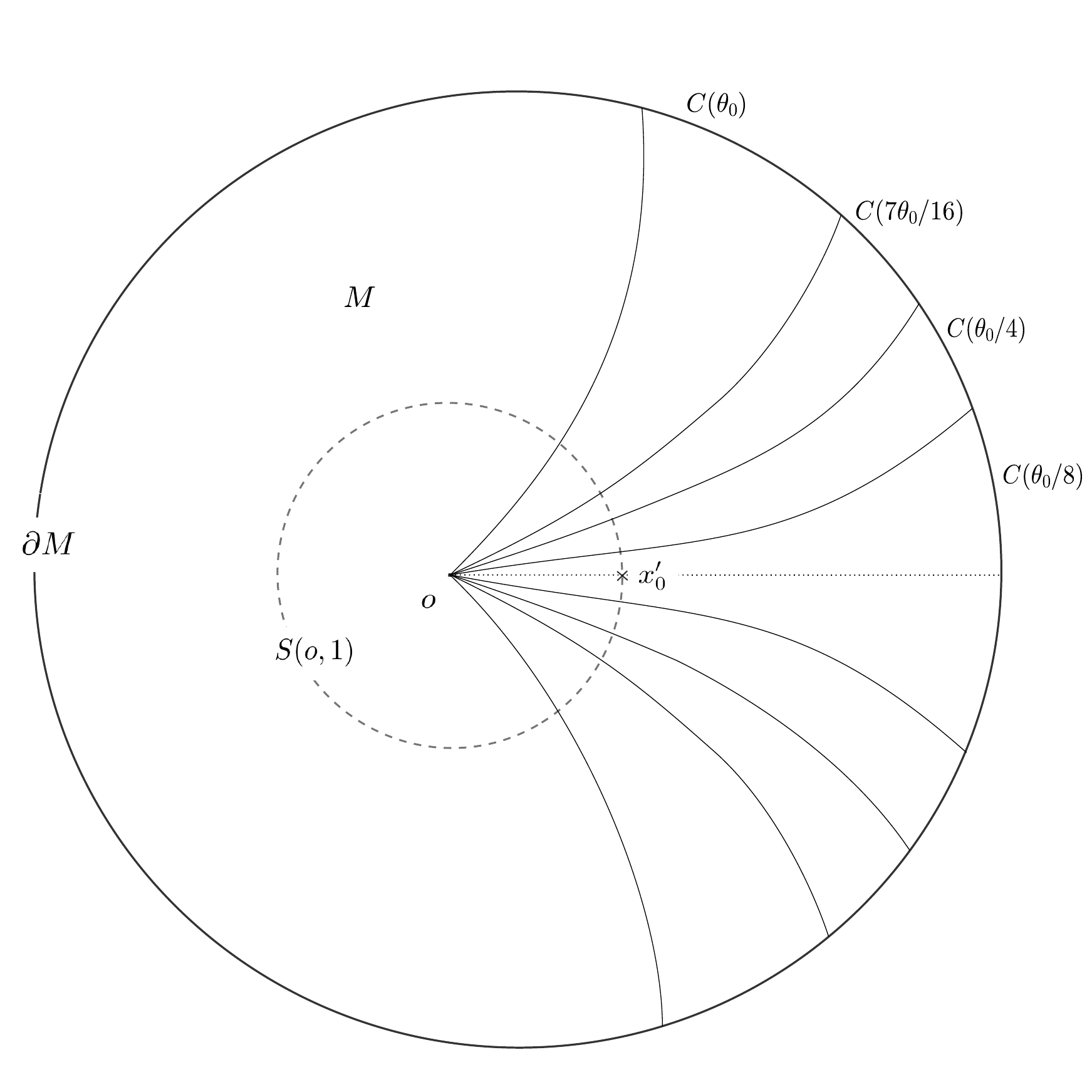}
	\caption{}
	\label{fig:drawing}
\end{figure}

\medskip

Now let us define, 
\begin{equation*}    
g(x):=c_{12}\:\theta^{-1}_0e^{-a\rho(x)}+\mathscr S (\varphi)(x) \text{ on } M,
\end{equation*}
where $\mathscr S (\varphi)$ is the Anderson-Schoen convolution of $\varphi$ and using Lemma \ref{AS_convolution_lemma}, $c_{12}$ is chosen large enough so that
\begin{equation*}
\Delta g\le 0\:,\: \text{on } C(\theta_0/4).
\end{equation*}

We next note that as $\overline{\varphi}$ (the extension of $\varphi$ along the radial geodesic rays) vanishes in the cone $C(\theta_0/8)$, by Lemma \ref{AS_convolution_lemma}, it follows that for all $x \in C(\theta_0/8)$,
\begin{equation*}
\mathscr S (\varphi)(x) \le \frac{c_{13}\: e^{-a \rho(x)}}{\theta_0} \:.
\end{equation*}
Thus there exists a positive constant $c_{14}$ such that for all $x \in C(\theta_0/8)$,
\begin{equation} \label{Lemma_pf_eq5}
g(x) \le \frac{c_{14}\: e^{-a \rho(x)}}{\theta_0} \:.
\end{equation}

We now want to prove the following estimate  for all $x \in \partial T(\theta_0/4,1) \cup (C(\theta_0/4) \cap S(o,1))$,
\begin{equation} \label{Lemma_pf_eq6}
 h(x) \le c_{15}\:e^{c_9R_0}\: h(x'_0)\:g(x)\:.
\end{equation}
In view of (\ref{Lemma_pf_eq4}), in order to prove (\ref{Lemma_pf_eq6}), it suffices to show that there exists $R_1\ge 1$, such that
\begin{itemize}
\item[(i)] $\mathscr S (\varphi) \ge 1/2$ on $\partial T(\theta_0/4,R_1)$ and
\item[(ii)] $c_{12}\:\theta^{-1}_0e^{-a\rho(x)} \ge 1/2$ for $x \in \overline{C(\theta_0/4)} \cap \overline{B(o,R_1)}$.
\end{itemize}

\medskip

We note from the definition of $\mathscr S (\varphi)$ that to get $(i)$, it suffices to have 
\begin{equation*}
\overline{\varphi} \ge 1/2\:,\: \text{ on } B(x,\sqrt{2})\:,
\end{equation*}
for all $x \in \partial T(\theta_0/4,R_1)$, which in turn follows (by the definition of $\varphi$) if the angle subtended by the ball $B(x,\sqrt{2})$ at $o$ (as defined in Lemma \ref{ball_angle_lemma}), $\omega(\rho(x))$ satisfies
\begin{equation*}
\omega(\rho(x)) \le \theta_0/16 \:.
\end{equation*}
Now by Lemma \ref{ball_angle_lemma}, we have 
\begin{equation*}
\omega(\rho(x)) \le c_{16}\: e^{-a \rho(x)}\:,
\end{equation*}
where $c_{16}$ is the constant appearing in the conclusion of Lemma \ref{ball_angle_lemma}. We further note that for $x \in \partial T(\theta_0/4,R_1)$, $e^{-a \rho(x)} \le e^{-a R_1}$ and thus we obtain a sufficient condition 
\begin{equation*}
c_{16} e^{-a R_1} \le \theta_0/16 \:,
\end{equation*}
which is equivalent to the following:
\begin{equation} \label{Lemma_pf_eq7}
R_1 \ge \frac{1}{a} \log \left(\frac{16c_{16}}{\theta_0}\right)\:.
\end{equation}

Now to get $(ii)$, we make the simple observation that for $x \in \overline{B(o,R_1)}$, $e^{-a \rho(x)} \ge e^{-a R_1}$ and hence we note the following sufficient condition:
\begin{equation*} 
c_{12}\:\theta^{-1}_0e^{-a R_1} \ge 1/2\:,
\end{equation*}
which in turn is equivalent to:
\begin{equation} \label{Lemma_pf_eq8}
R_1 \le \frac{1}{a} \log \left(\frac{2c_{12}}{\theta_0}\right)\:.
\end{equation}
Then we note that 
\begin{equation*}
R_1 = \frac{1}{2a} \log \left(\frac{16c_{16}}{\theta_0}\right) + \frac{1}{2a} \log \left(\frac{2c_{12}}{\theta_0}\right)\:,
\end{equation*}
satisfies both (\ref{Lemma_pf_eq7}) and (\ref{Lemma_pf_eq8}) and this completes the proof of $(i)$ and $(ii)$. Thus by taking $c_{15}=2c_{10}$, we get the estimate (\ref{Lemma_pf_eq6}).

\medskip

Now as $h$ vanishes on $\overline{C(\theta_0/4)} \cap \partial M$, and $g$ is a non-negative superharmonic function in $C(\theta_0/4)$ with boundary values $\varphi$ on $\overline{C(\theta_0/4)} \cap \partial M$, (\ref{Lemma_pf_eq6}) and the maximum principle imply
\begin{equation} \label{Lemma_pf_eq9}
 h(x) \le c_{15}\:e^{c_9R_0}\: h(x'_0)\:g(x)\:,
\end{equation}
for all $x \in T(\theta_0/4,1)$. Thus combining (\ref{Lemma_pf_eq5}), (\ref{Lemma_pf_eq9}) and using the expression (\ref{Lemma_pf_eq2}), we get the desired estimate:
\begin{equation*}
h(x) \le c_{17} {\left(\frac{1}{\theta_0}\right)}^{c_{18}} e^{-a \rho(x)} h(x'_0) \:,
\end{equation*}
for all $x \in T(\theta_0/8,1)$.
\end{proof}

\section{Riemannian angle and Gromov product}
In this section, we obtain useful relations between Riemannian angles and suitable Gromov products. These geometric inequalities are familiar to experts. Nevertheless we present them here for the sake of completeness. 

\medskip

We first obtain the following lower bound for the Gromov product between two points on the Gromov boundary $\partial M$, in terms of the Riemannian angle between them:  
	\begin{lemma} \label{angle_gromov_rel1}
		Let $M$ be a Hadamard manifold of dimension $\ge 2$, with sectional curvature bounds $-b^2 \le K_M \le 0$. Fix $x \in M$. Let $\eta , \xi \in \partial M$ such that $\eta \ne \xi$ and let $\theta$ denote the Riemannian angle between $\eta$ and $\xi$ subtended at $x$. Then one has,
		\begin{equation*}
		e^{-b{(\xi | \eta )}_x} \le \sin \frac{\theta}{2} \ .
		\end{equation*}
	\end{lemma}
	\begin{proof}
		Let $\gamma_1$ and $\gamma_2$ be the two geodesic rays starting at $x$ that hits $\partial M$ at $\eta$ and $\xi$ respectively. Now for $t\in (0,\infty)$ we consider the geodesic triangle $\Delta(x,\gamma_1(t),\gamma_2(t))$. Then let $\theta^t_b$ be the angle corresponding to $\theta$ in the comparison triangle $\overline{\Delta}(x,\gamma_1(t),\gamma_2(t))$ in $\mathbb H^2(-b^2)$. The curvature pinching condition yields $\theta^t_b \le \theta$ and hence
		\begin{equation*}
	\sin\frac{\theta^t_b}{2}\le \sin\frac{\theta}{2} \quad \text{for all} \quad t\in (0,\infty)\:.
		\end{equation*} 
		Now by the hyperbolic law of cosines,
		\begin{equation*}
		\sin^2\frac{\theta^t_b}{2}=\frac{\cosh bd(\gamma_1(t),\gamma_2(t))-\cosh b(d(\gamma_1(t),x)-d(\gamma_2(t),x))}{2\sinh bd(\gamma_1(t),x)\sinh bd(\gamma_2(t),x)}\:.
		\end{equation*}
		Then for large $t$,
		\begin{align*}
		\frac{\cosh bd(\gamma_1(t),\gamma_2(t))}{2\sinh bd(\gamma_1(t),x)\sinh bd(\gamma_2(t),x)}&\ge e^{-b\{d(\gamma_2(t),x)+d(\gamma_1(t),x)-d(\gamma_2(t),\gamma_1(t))\}}\\
		& =e^{-2b{\left(\gamma_2(t)|\gamma_1(t)\right)}_x}\\
		& \ge e^{-2b{(\xi|\eta)}_x}\:,
		\end{align*}
		and 
		\begin{align*}
		\frac{\cosh b(d(\gamma_1(t),x)-d(\gamma_2(t),x))}{2\sinh bd(\gamma_1(t),x)\sinh bd(\gamma_2(t),x)}&=\frac{1}{2\sinh (bt)\sinh (bt)}
		\to 0 \quad \text{( as $t\to\infty )$} \:.
		\end{align*}
		Hence from the above we have 
		\begin{equation*}
		e^{-b{(\xi | \eta )}_x} \le \sin \frac{\theta}{2} \:.
		\end{equation*}
		
	\end{proof}
	Next we consider the case when one point is in $M$, while the other one is on $\partial M$.
	\begin{lemma} \label{angle_gromov_rel2}
		Let $M$ be a Hadamard manifold of dimension $\ge 2$, with sectional curvature bounds $-b^2 \le K_M \le 0$. Fix $x \in M$. Let $y \in M$ and $\xi \in \partial M$ such that $x,y$ and $\xi$ are not collinear. Then the Riemannian angle between $y$ and $\xi$ subtended at $x$, say $\theta$, satisfies 
		\begin{equation*}
		e^{-2b{(y|\xi)}_x} - e^{-2bd(x,y)} \le \sin^2 \frac{\theta}{2} \:.
		\end{equation*}
	\end{lemma}
	\begin{proof}
		Let $\gamma$ denote the geodesic ray starting from $x$ and hitting $\partial M$ at $\xi$. Now for any $t\in (0,\infty)$, we consider the geodesic triangle $\Delta(x,y,\gamma_2(t))$. Then let $\theta^t_b$ be the angle corresponding to $\theta$ in the comparison triangle $\overline{\Delta}(x,y,\gamma(t))$ in $\mathbb H^2(-b^2)$. Again as in the previous lemma the curvature pinching condition gives us 
		\begin{equation*}
		\sin\frac{\theta^t_b}{2}\le \sin\frac{\theta}{2} \quad \text{for all} \quad t\in (0,\infty)\:.
		\end{equation*} 
		Then by the hyperbolic law of cosines,
		\begin{equation*}
		\sin^2\frac{\theta^t_b}{2}=\frac{\cosh bd(y,\gamma(t))-\cosh b(d(y,x)-d(\gamma(t),x))}{2\sinh bd(y,x)\sinh bd(\gamma(t),x)}
		\end{equation*}
		Now
		\begin{align*}
		&\lim_{t\to\infty} \frac{\cosh bd(y,\gamma(t))}{2\sinh bd(y,x)\sinh bd(\gamma(t),x)}\\
		=&\lim_{t\to\infty} \frac{e^{bd(y,\gamma(t))}+e^{-bd(y,\gamma(t))}}{(e^{bd(y,x)}-e^{-bd(y,x)})(e^{bd(x,\gamma(t))}-e^{-bd(x,\gamma(t))})}\\
		=&\frac{e^{-2b(y|\xi)_x}}{1-e^{-2bd(x,y)}}\:,
		\end{align*}
		and
		\begin{align*}
		&\lim_{t\to\infty} \frac{\cosh b(d(y,x)-d(\gamma(t),x))}{2\sinh bd(y,x)\sinh bd(\gamma(t),x)}\\
		=&\lim_{t\to\infty} \frac{e^{b(d(y,x)-d(\gamma(t),x))}+e^{-b(d(y,x)-d(\gamma(t),x))}}{(e^{bd(y,x)}-e^{-bd(y,x)})(e^{bd(x,\gamma(t))}-e^{-bd(x,\gamma(t))})}\\
		=&\frac{e^{-2bd(x,y)}}{1-e^{-2bd(x,y)}}\:.
		\end{align*}
		Hence,
		\begin{equation*}
		e^{-2b{(y|\xi)}_x} - e^{-2bd(x,y)} \le\frac{e^{-2b{(y|\xi)}_x} - e^{-2bd(x,y)}}{1-e^{-2bd(x,y)}} \le \sin^2 \frac{\theta}{2} \:.
		\end{equation*}
		\end{proof} 
Now in the case when the manifold satisfies $K_M \le -a^2$ instead, doing angle comparison in $\mathbb{H}^2(-a^2)$ and then proceeding exactly as in Lemma \ref{angle_gromov_rel2}, we get the following:
\begin{lemma} \label{angle_gromov_rel3}
		Let $M$ be a Hadamard manifold of dimension $\ge 2$, with sectional curvature $K_M \le -a^2$. Fix $x \in M$. Let $y \in M$ and $\xi \in \partial M$ such that $x,y$ and $\xi$ are not collinear. Then the Riemannian angle between $y$ and $\xi$ subtended at $x$, say $\theta$, satisfies 
		\begin{equation*}
		\sin^2 \frac{\theta}{2} \le \frac{e^{-2a{(y|\xi)}_x} - e^{-2ad(x,y)}}{1-e^{-2ad(x,y)}}\:.
		\end{equation*}
	\end{lemma}

	\section{Lower bound for the Poisson kernel}
	The aim of this section is to prove the lower bound in (\ref{Poisson_estimate}). We first choose and fix $\xi \in \partial M$. Next consider the geodesic ray that joins $\xi$ to $o$, and then extend it beyond $o$. This extended bi-infinite geodesic ray will hit $\partial M$ at another point, which we denote by $\xi^-$. Then depending on the location of $x \in M$, we break the proof into two cases:
	\begin{itemize}
		\item $x \in \mathscr{C}(\xi^-,\pi/4)$ 
		\item $x \notin \mathscr{C}(\xi^-,\pi/4)$ \:.
	\end{itemize}
	\subsection{Inside the cone neighbourhood $\mathscr{C}(\xi^-,\pi/4)$ :}
	As $x$ is in the cone neighbourhood $\mathscr{C}(\xi^-,\pi/4)$, there exists $c_1 >0$, such that
	\begin{equation*}
	{(x|\xi)}_o \le c_1 \:.
	\end{equation*}
	Combining the above with the fact that 
	\begin{equation*}
	{(o|\xi)}_x - \rho (x) = - {(x|\xi)}_o \:,
	\end{equation*}
	it follows that
	\begin{equation*}
	{(o|\xi)}_x \ge \rho(x) -c_1 \:.
	\end{equation*}
	Hence for any $\tau > 0$, one has
	\begin{equation} \label{lower_eq1}
	e^{-\tau {(o|\xi)}_x} \: e^{a \rho(x)} \le e^{\tau c_1} \:e^{-(\tau - a) \rho(x)} \:.
	\end{equation}

\medskip

Now by Harnack-Yau (Lemma \ref{Har-Yau_lemma}), there exists $c_2 >0$, such that the Poisson kernel satisfies
	\begin{equation} \label{lower_eq2}
	e^{-c_2 \rho(x)} \le P(x,\xi) \:.
	\end{equation}
	
\medskip
	
We also note that for $c_3 \ge c_2 + a$, one has
	\begin{equation} \label{lower_eq3}
	e^{-(c_3 - a) \rho(x)} \le e^{-c_2 \rho(x)} \:.
	\end{equation}
	Then putting $c_4=e^{c_1c_3}$ and combining (\ref{lower_eq1})-(\ref{lower_eq3}), it follows that
	\begin{equation*}
	\frac{1}{c_4}\: e^{-c_3{(o|\xi)}_x} \: e^{a \rho(x)} \le \:e^{-(c_3 - a) \rho(x)} \le e^{-c_2 \rho(x)} \le P(x,\xi) \:.
	\end{equation*}
	\subsection{Outside the cone neighbourhood $\mathscr{C}(\xi^-,\pi/4)$ :}
	We consider the geodesic segment joining $x$ to $o$, say $\gamma$, and also the geodesic segment in the reverse direction that joins $o$ to $x$, say $\tilde{\gamma}$. First extending $\gamma$ (beyond $o$) we obtain a unique point on $\partial M$, say $\eta_x$. We next extend $\tilde{\gamma}$ beyond $x$ and consider the unique point $\tilde{x}$ on $S(x,1)$ lying on the extended geodesic segment. Then as $x$ lies outside the cone neighbourhood $\mathscr{C}(\xi^-,\pi/4)$, $\eta_x$ also stays outside a fixed cone neighbourhood of $\xi$ and hence, there exists $c_5 > 0$, such that
	\begin{equation*}
	{(\xi | \eta_x)}_o \le c_5 \:.
	\end{equation*}  
	Now combining this with the fact that 
	\begin{equation*}
	{(\xi|\eta_x)}_{\tilde{x}} = {(o|\xi)}_{\tilde{x}}  + {(\xi|\eta_x)}_o \:,
	\end{equation*}
	it follows that
	\begin{equation*}
	{(\xi|\eta_x)}_{\tilde{x}} \le {(o|\xi)}_{\tilde{x}}  + c_5 \:.
	\end{equation*}
	Hence for any $\tau > 0$, one has
	\begin{equation} \label{lower_eq4}
	e^{-\tau {(o|\xi)}_{\tilde{x}}} \le e^{\tau c_5} \: e^{-\tau {(\xi|\eta_x)}_{\tilde{x}}} \:.
	\end{equation}
	Let $\theta$ be the Riemannian angle between $\xi$ and $\eta_x$ subtended at ${\tilde{x}}$ (see \Cref{fig:Lower bound}). By Lemma \ref{angle_gromov_rel1}, it follows that
	\begin{equation} \label{lower_eq5}
	e^{-b{(\xi|\eta_x)}_{\tilde{x}}} \le \sin \frac{\theta}{2} \:.
	\end{equation}
	
	\begin{figure}
		\centering
		\includegraphics[width=.7\linewidth]{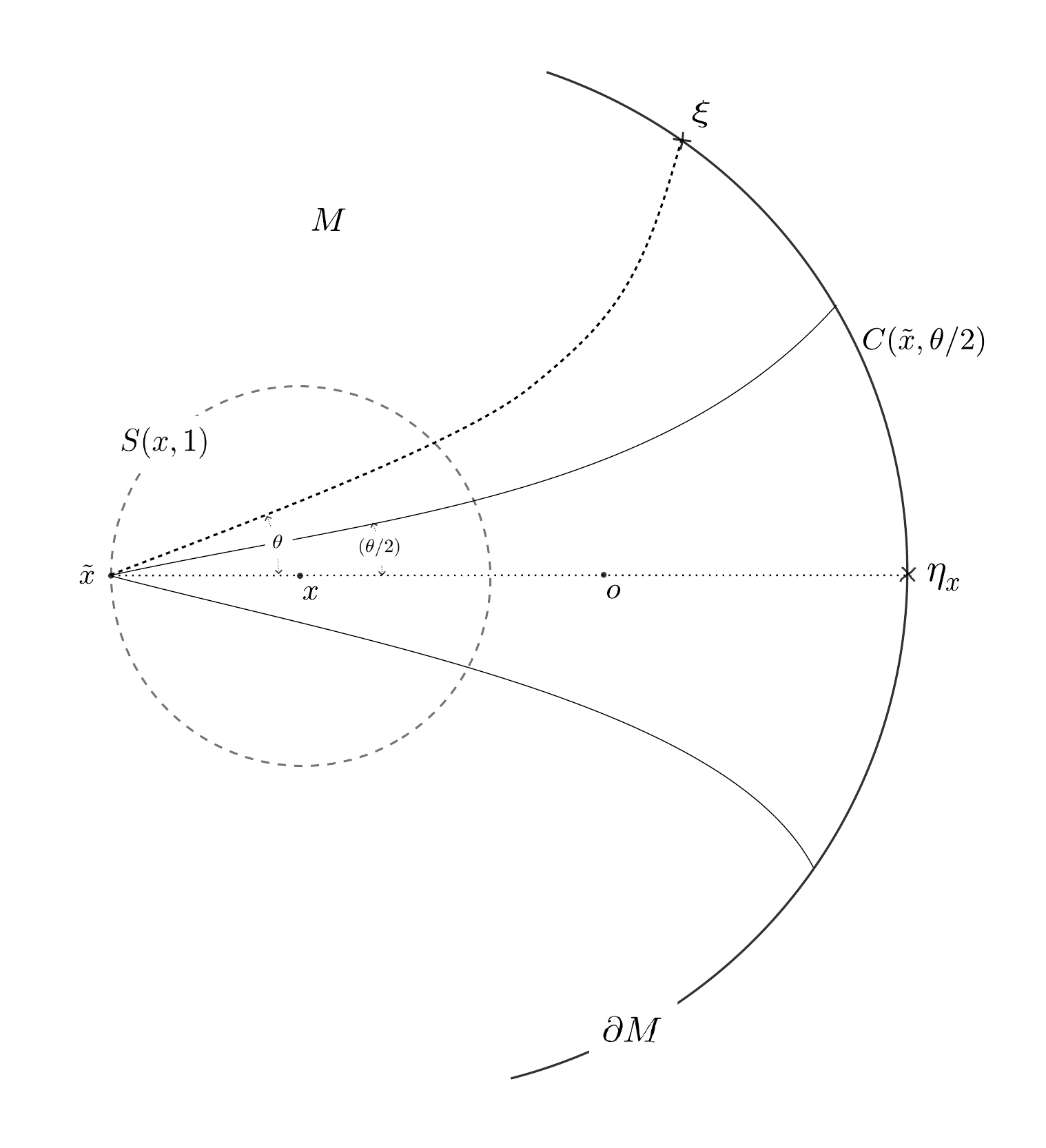}
		\caption{}
		\label{fig:Lower bound}
	\end{figure}
	
	Now by applying Lemma \ref{estimate_lemma} on the Poisson kernel for the cone with vertex ${\tilde{x}}$ and aperture $\theta/2$ with respect to the axis being the geodesic ray joining ${\tilde{x}}$ to $\eta_x$, we get that there exist two positive constants $c_6$ and $c_7$, such that
	\begin{equation} \label{lower_eq6}
	\frac{1}{c_6} \:\theta^{c_7} e^{a \rho(x)} \le P(x,\xi) \:.
	\end{equation}
Note that above we have used that $P(o,\cdot)\equiv 1$. Now putting $\tau=bc_7$ in (\ref{lower_eq4}) and then combining it with (\ref{lower_eq5}), one has
	\begin{eqnarray} \label{lower_eq7}
	e^{-bc_7{(o|\xi)}_{\tilde{x}}} &\le & e^{bc_5c_7} e^{-bc_7{(\xi|\eta_x)}_{\tilde{x}}} \nonumber \\
	& \le & e^{bc_5c_7} {\left(\sin \frac{\theta}{2}\right)}^{c_7} \nonumber \\
	& \le & \left(\frac{e^{bc_5c_7}}{2^{c_7}}\right) \theta^{c_7} \nonumber \\
	&=& c_8 \theta^{c_7} \:.
	\end{eqnarray}
	Then plugging (\ref{lower_eq7}) in (\ref{lower_eq6}), it follows that
	\begin{equation} \label{lower_eq8}
	\frac{1}{c_6\: c_8}\: e^{-bc_7{(o|\xi)}_{\tilde{x}}}\: e^{a \rho(x)} \le \frac{1}{c_6}\: \theta^{c_7}\: e^{a \rho(x)} \le P(x,\xi) \:. 
	\end{equation}
Now as $d(x,\tilde{x})=1$, we have by triangle inequality,
\begin{equation} \label{lower_eq9}
{\left(o|\xi\right)}_{\tilde{x}} = {\left(o|\xi\right)}_{x} + {\left(x|\xi\right)}_{\tilde{x}} \le {\left(o|\xi\right)}_{x} + d(x,\tilde{x}) = {\left(o|\xi\right)}_{x} + 1\:.
\end{equation}
Finally, plugging (\ref{lower_eq9}) in (\ref{lower_eq8}), we get that
\begin{equation*}
\left(\frac{e^{-bc_7}}{c_6\: c_8}\right) e^{-bc_7{(o|\xi)}_{x}}\: e^{a \rho(x)}  \le P(x,\xi) \:.
\end{equation*}

	This completes the proof for the lower bound in (\ref{Poisson_estimate}).
	\begin{remark}
		When $x$ lies in a Stolz angle based at $\xi$, $P(x,\xi)$ grows exponentially. This can be seen from (\ref{Poisson_estimate}), as in this case
		\begin{equation*}
		{(o|\xi)}_x \le c' \:,\text{ for some }c' > 0 \:,
		\end{equation*}
		and hence the lower bound of (\ref{Poisson_estimate}) gives,
		\begin{equation*}
		\left(\frac{e^{-k_1 c'}}{c}\right) e^{a \rho (x)} \le P(x,\xi) \:.
		\end{equation*}
	\end{remark}
	\section{Upper bound for the Poisson kernel}
	The aim of this section is to prove the upper bound in (\ref{Poisson_estimate}). As in section $5$, we first choose and fix $\xi \in \partial M$, and then break the proof into two cases: 
	\begin{itemize}
		\item $\rho(x)-3 \le {(x|\xi)}_o$ 
		\item $\rho(x)-3 \ge {(x|\xi)}_o$ \:.
	\end{itemize} 
	Note that the above two cases correspond to whether $x$ lies in a Stolz angle based at $\xi$ or not respectively.
	\subsection{Inside a Stolz angle based at $\xi$ :}
	By Harnack-Yau (Lemma \ref{Har-Yau_lemma}), there exists $c_1 > 0$, such that 
	\begin{equation} \label{upper_eq1}
	P(x,\xi) \le e^{c_1 \rho(x)} \:.
	\end{equation}

\medskip

	Now as 
	\begin{equation*}
	\rho(x)-3 \le {(x|\xi)}_o \:,
	\end{equation*}
	it follows that for any $\tau >0$,
	\begin{equation} \label{upper_eq2}
	e^{-3 \tau} e^{(\tau - a) \rho (x)} \le e^{\tau {(x|\xi)}_o}\: e^{-a \rho (x)} \:.
	\end{equation}

\medskip

We also note that for $c_2 \ge c_1 + a$,
	\begin{equation} \label{upper_eq3}
	e^{c_1 \rho(x)} \le e^{(c_2-a)\rho(x)} \:.
	\end{equation}
	Hence putting $c_3=e^{3c_2}$ and combining (\ref{upper_eq1})-(\ref{upper_eq3}), we get
	\begin{equation*}
	P(x,\xi) \le e^{c_1 \rho (x)} \le e^{(c_2-a)\rho(x)} \le c_3 \:e^{c_2 {(x|\xi)}_o}\: e^{-a \rho (x)} \:.
	\end{equation*}
	\subsection{Outside a Stolz angle based at $\xi$ :}
We consider the geodesic segment joining $x$ to $o$, say $\gamma$. Extending $\gamma$ beyond $o$, we obtain the unique point $\tilde{o}$ on $S(o,1)$ lying on the extended geodesic segment. Also let $\theta$ be the Riemannian angle between $x$ and $\xi$, subtended at $\tilde{o}$. Now by applying Lemma \ref{estimate_lemma} on the Poisson kernel for the cone with vertex $\tilde{o}$ and aperture $\theta/2$, with respect to the axis being the geodesic ray which is the infinite extension of the geodesic segment joining $\tilde{o}$ to $x$ (see \Cref{fig:Upper bound}), we get that there exist two positive constants $c_4$ and $c_5$ such that
	\begin{equation} \label{upper_eq4}
	P(x,\xi) \le c_4 {\left(\frac{1}{\theta}\right)}^{c_5} e^{-a \rho(x)} \:.
	\end{equation}
Note that above we have again used that $P(o,\cdot)\equiv 1$. Next, by Lemma \ref{angle_gromov_rel2}, one has
	\begin{equation} \label{upper_eq5}
	\sin^2 \frac{\theta}{2} \ge e^{-2b{(x|\xi)}_{\tilde{o}}} - e^{-2b d(\tilde{o},x)} \:.
	\end{equation}
	\begin{figure}
		\centering
		\includegraphics[width=.7\linewidth]{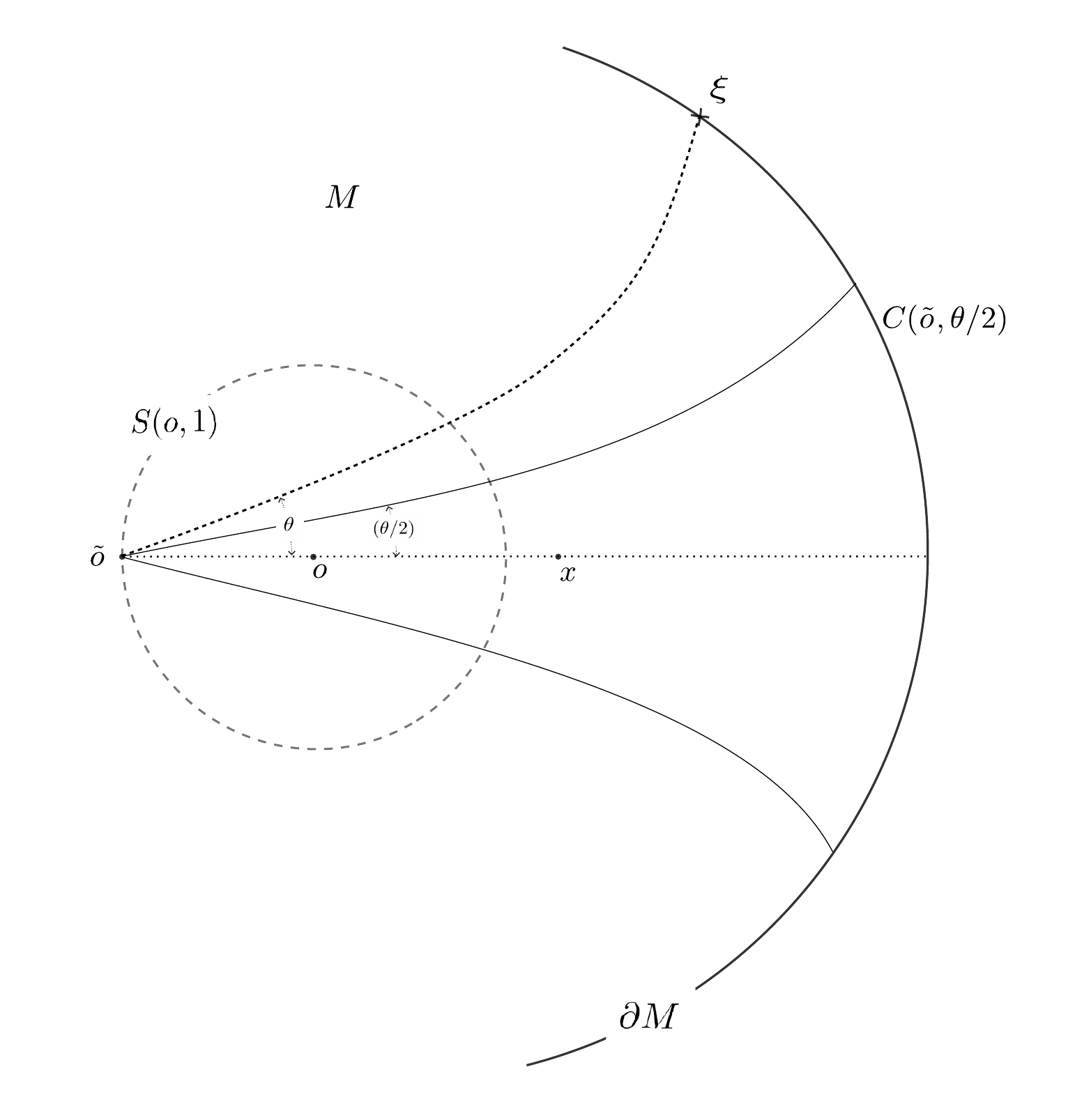}
		\caption{}
		\label{fig:Upper bound}
	\end{figure}
	Now as $d(\tilde{o},o)=1$, we have by triangle inequality,
\begin{equation*}
{\left(x|\xi\right)}_{\tilde{o}}= {\left(x|\xi\right)}_{o} + {\left(o|\xi\right)}_{\tilde{o}}    \le {\left(x|\xi\right)}_{o} + \rho(\tilde{o}) = {\left(x|\xi\right)}_{o} + 1\:,
\end{equation*}     
and thus,
\begin{equation} \label{upper_eq6}
e^{-2b{(x|\xi)}_{\tilde{o}}} \ge e^{-2b}\: e^{-2b{(x|\xi)}_{o}}\:.
\end{equation}	
Similarly, one also has
	\begin{equation} \label{upper_eq7}
e^{-2b d(\tilde{o},x)} \le e^{2b}\: e^{-2b \rho(x)}\:.
\end{equation}
Then plugging (\ref{upper_eq6}) and (\ref{upper_eq7}) in (\ref{upper_eq5}), we get that
\begin{equation} \label{upper_eq8}
	\sin^2 \frac{\theta}{2} \ge e^{-2b}\: e^{-2b{(x|\xi)}_{o}} - e^{2b}\: e^{-2b \rho(x)} \:.
	\end{equation}	
	Now using the hypothesis 
	\begin{equation*}
	\rho(x)-3 \ge {(x|\xi)}_o \:,
	\end{equation*}
	in (\ref{upper_eq8}), one has
	\begin{equation} \label{upper_eq9}
	{\left(\frac{\theta}{2}\right)}^2 \ge \sin^2 \left(\frac{\theta}{2}\right) \ge e^{-2b{(x|\xi)}_o} \: \left(e^{-2b}-e^{-4b}\right)\:.
	\end{equation}
	Finally plugging (\ref{upper_eq9}) in (\ref{upper_eq4}), it follows that
	\begin{equation*}
	P(x,\xi) \le \frac{c_4}{2^{c_5} {\left(e^{-2b}-e^{-4b}\right)}^{c_5/2}}  \:e^{bc_5{(x|\xi)}_o}\:e^{-a \rho(x)} \:.
	\end{equation*}
This completes the proof of Theorem \ref{Poisson_estimate_thm}.

\section{Some applications}
We now present some quantitative versions of concentration and convergence of harmonic measures, as applications of the pointwise estimates of the Poisson kernel.

\medskip

We first introduce the notion of spherical caps. For $\xi \in \partial M$ and $\alpha \in (0,\pi)$, {\em the spherical cap with center $\xi$ and radius $\alpha$}, denoted by $\mathcal{C}(\xi,\alpha)$ is the ball with center $\xi$ and radius $\alpha$, in the metric space $\left(\partial M,\angle_o\right)$. Equivalently, $\mathcal{C}(\xi,\alpha)$ is the intersection of the closure (in $\overline{M}$) of the cone with vertex $o$ and aperture $\alpha$, with respect to the axis being the geodesic ray joining $o$ to $\xi$, with the boundary at infinity $\partial M$.

\subsection{Concentration of harmonic measure:} As a simple consequence of the fact that the Dirichlet problem at infinity is solvable \cite{AS}, one has for $\xi \in \partial M$ and $x \in M$, the weak convergence of measures:
\begin{equation*}
\mu_{x} \to \delta_\xi\:,\text{ as } x \to \xi \text{ in cone topology.}
\end{equation*}
Our first result is to obtain a quantitative version of the above convergence by means of estimating the concentration of harmonic measures on spherical caps, in terms of their aperture and an exponential decay in the distance of $x$ from $o$. More precisely:
\begin{proposition} \label{application_thm1}
\label{cor_lower_bound_har_meas}
Let $M,K$ and $C$ be as in Theorem \ref{Poisson_estimate_thm}. Choose and fix $\xi \in \partial M$. Let $\alpha \in (0,\pi)$ and $x$ lie on the geodesic ray joining $o$ to $\xi$. Then 
\begin{equation} \label{lower_bound_har_meas_eq}
\mu_{x}\left(\partial M \setminus \mathcal{C}(\xi,\alpha)\right) \le \frac{C\: e^{-a \rho(x)}}{{\left(\sin (\alpha/2)\right)}^{2K/a}} \:.
\end{equation}  
\end{proposition} 
\begin{proof}
Using the fact that the Poisson kernel $P(x,\cdot)$ is the Radon-Nykodym derivative of $\mu_x$ with respect to $\mu_o$ and then applying the upper bound in the estimate of the Poisson kernel (\ref{Poisson_estimate}), we get that
\begin{eqnarray} \label{application_thm1_eq1}
\mu_{x} \left(\partial M \setminus \mathcal{C}(\xi,\alpha)\right) 
&=& \int_{\partial M \setminus \mathcal{C}(\xi,\alpha)} P(x,\eta)\:d\mu_o(\eta) \nonumber\\
&\le & C \:  e^{-a \rho(x)} \int_{\partial M \setminus \mathcal{C}(\xi,\alpha)} e^{2K{(x|\eta)}_o}\:d\mu_o(\eta)\:.
\end{eqnarray}
Now by Lemma \ref{angle_gromov_rel3}, for any $\eta \ne \xi \in \partial M$,
\begin{equation*}
\sin^2 \left(\frac{\angle_o(\xi,\eta)}{2}\right) = \sin^2 \left(\frac{\angle_o(x,\eta)}{2}\right) \le \frac{e^{-2a{(x|\eta)}_o} - e^{-2a\rho(x)}}{1-e^{-2a\rho(x)}}\:,
\end{equation*}
which in turn implies that
\begin{eqnarray} \label{application_thm1_eq2}
e^{2a{(x|\eta)}_o} & \le & \frac{1}{\left(1-e^{-2a\rho(x)}\right) \sin^2 \left(\frac{\angle_o(\xi,\eta)}{2}\right) + e^{-2a\rho(x)}} \nonumber\\
&=& \frac{1}{\sin^2 \left(\frac{\angle_o(\xi,\eta)}{2}\right) + e^{-2a\rho(x)}\cos^2 \left(\frac{\angle_o(\xi,\eta)}{2}\right)} \nonumber\\
&\le & \frac{1}{\sin^2 \left(\frac{\angle_o(\xi,\eta)}{2}\right)}\:. 
\end{eqnarray}
Then plugging (\ref{application_thm1_eq2}) in (\ref{application_thm1_eq1}), we get that
\begin{eqnarray*}
\mu_{x} \left(\partial M \setminus \mathcal{C}(\xi,\alpha)\right) & \le & C\: e^{-a \rho(x)} \int_{\partial M \setminus \mathcal{C}(\xi,\alpha)} \frac{1}{{\left(\sin(\angle_o(\xi,\eta)/2)\right)}^{2K/a}}\:d\mu_o(\eta) \nonumber\\
&\le & \frac{C\: e^{-a \rho(x)}}{{\left(\sin (\alpha/2)\right)}^{2K/a}} \:.
\end{eqnarray*}
\end{proof}

\subsection{Convergence of harmonic measures:}
For $\xi \in \partial M$, let $\gamma_\xi$ denote the radial geodesic ray such that $\gamma_\xi(0)=o$ and $\gamma_\xi(\infty)=\xi$. We now consider the radial projection map from the finite sphere of radius $R$ with center $o$, to the boundary at infinity:
\begin{eqnarray*}
\pi_R &:& S(o,R) \to \partial M \\
&& \gamma_\xi(R) \mapsto \gamma_\xi(\infty)=\xi\:.
\end{eqnarray*}
Then by taking the pushforward of the harmonic measures $\mu_{o,R}$ on $S(o,R)$ onto $\partial M$ under $\pi_R$, we get a family of probability measures $\{{(\pi_R)}_* \mu_{o,R}\}_{R>0}$ on $\partial M$. In the special case of rank one Riemannian symmetric spaces of noncompact type or more generally, negatively curved Harmonic manifolds, these pushforwarded measures coincide with $\mu_o$, the harmonic measure at infinity. In the full generality of Hadamard manifolds of pinched negative curvature however, this is no longer true, in fact, they may not even be mutually absolutely continuous with respect to $\mu_o$. However by soft 
arguments it is possible to show the following weak convergence:
\begin{equation*}
{(\pi_R)}_* \mu_{o,R} \to \mu_o\:,\text{ as } R \to \infty\:.
\end{equation*}
Our next result is to obtain an exponential rate of the above convergence in the radius $R$, in the case of H\"older continuous functions:
\begin{theorem} \label{application_thm2}
Let $M$ and $K$ be as in Theorem \ref{Poisson_estimate_thm}. Then there exists a positive constant $c_1$ depending solely on $a,b$ and $n$, such that for sufficiently large $R>0$ and any positive $C^{\beta}$-H\"older continuous function $f$ on $(\partial M,\angle_o)$, we have
\begin{equation*}
\left|\int_{\partial M} f\: d\mu_o - \int_{\partial M} f\: d\left({(\pi_R)}_* \mu_{o,R}\right)\right| \le c_1 \: {\|f\|}_{C^\beta}\: e^{-\lambda R}\:,
\end{equation*}
where
\begin{equation*}
\lambda=\frac{a^2 \beta}{a \beta +2K}\:.
\end{equation*}
\end{theorem}
\begin{proof}
By the probabilistic interpretation of harmonic measures, the conditional expectation of Brownian motion and the Markov property, we have
\begin{equation} \label{application_thm2_pf_eq1}
\int_{\partial M} f\:d\mu_o= \int_{S(o,R)} \left(\int_{\partial M} f\:d\mu_x \right)d\mu_{o,R}(x)\:.
\end{equation}
Now for any $\alpha \in (0,\pi)$, we write 
\begin{equation} \label{application_thm2_pf_eq2}
\int_{\partial M} f\:d\mu_x = \int_{\mathcal{C}(\pi_R(x),\alpha)} f\:d\mu_x + \int_{\partial M \setminus \mathcal{C}(\pi_R(x),\alpha)} f\:d\mu_x \:.
\end{equation}
By the concentration of harmonic measures on spherical caps (\ref{lower_bound_har_meas_eq}) we have for $C$ as in the conclusion of Proposition \ref{application_thm1},
\begin{equation}  \label{application_thm2_pf_eq3}
\left|\int_{\partial M \setminus \mathcal{C}(\pi_R(x),\alpha)} f\:d\mu_x \right| \le \frac{C\:{\|f\|}_\infty\: e^{-a \rho(x)}}{{\left(\sin (\alpha/2)\right)}^{2K/a}}\:.
\end{equation}
Next by the $C^\beta$-H\"older regularity of $f$, we have the following estimate:
\begin{equation} \label{application_thm2_pf_eq4}
\int_{\mathcal{C}(\pi_R(x),\alpha)} f\:d\mu_x  = f\left(\pi_R(x)\right) \mu_x \left(\mathcal{C}(\pi_R(x),\alpha) \right) + E_1(f,\alpha,\beta) \:,
\end{equation}
with 
\begin{equation} \label{application_thm2_pf_eq5}
\left|E_1(f,\alpha,\beta)\right| \le {|f|}_{C^\beta}\: \alpha^\beta\:.
\end{equation}
Moreover, by the concentration of harmonic measures on spherical caps (\ref{lower_bound_har_meas_eq}) we have,
\begin{equation} \label{application_thm2_pf_eq6}
f\left(\pi_R(x)\right) \mu_x \left(\mathcal{C}(\pi_R(x),\alpha) \right) = f\left(\pi_R(x)\right)  + E_2(f,x,\alpha)\:,
\end{equation}
with
\begin{equation}\label{application_thm2_pf_eq7}
\left|E_2(f,x,\alpha)\right| \le \frac{C\:{\|f\|}_\infty\: e^{-a \rho(x)}}{{\left(\sin (\alpha/2)\right)}^{2K/a}}\:,
\end{equation}
where $C$ is as in the conclusion of Proposition \ref{application_thm1}. Thus plugging (\ref{application_thm2_pf_eq2})-(\ref{application_thm2_pf_eq7}) in (\ref{application_thm2_pf_eq1}), we get that
\begin{equation*}
\left|\int_{\partial M} f\: d\mu_o - \int_{\partial M} f\: d\left({(\pi_R)}_* \mu_{o,R}\right)\right| \le {|f|}_{C^\beta}\: \alpha^\beta + \frac{2C\:{\|f\|}_\infty\: e^{-a R}}{{\left(\sin (\alpha/2)\right)}^{2K/a}}\:.
\end{equation*}
Then by setting
\begin{equation*}
R = \frac{1}{a}\left(\frac{2K}{a}+\beta\right) \log\left(\frac{1}{\alpha}\right),
\end{equation*}
we have for
\begin{equation*}
\lambda=\frac{a^2 \beta}{a \beta +2K}\:,
\end{equation*}
and positive constants $c_2,c_3$ depending only on $a,b$ and $n$,
\begin{eqnarray*}
\left|\int_{\partial M} f\: d\mu_o - \int_{\partial M} f\: d\left({(\pi_R)}_* \mu_{o,R}\right)\right| & \le & \left({|f|}_{C^\beta} + c_2 \: {\|f\|}_\infty\right) e^{-\lambda R} \\
& \le & c_3 \: {\|f\|}_{C^\beta}\: e^{-\lambda R}\:.
\end{eqnarray*}
\end{proof}

\section*{Acknowledgements}
The second and the third authors are supported by research fellowships from Indian Statistical Institute.

\bibliographystyle{amsplain}

\end{document}